\documentclass[11pt,a4paper]{amsart}
\usepackage{amsmath,amsthm,amsfonts,graphicx,calc,float,fullpage}
\usepackage[numbers,sort&compress]{natbib}
\renewcommand{\baselinestretch}{1.12}
\usepackage[colorlinks=true,citecolor=black,linkcolor=black]{hyperref}
\newcommand{\doi}[1]{\href{http://dx.doi.org/#1}{\texttt{doi:#1}}}

\newcommand{\urlprefix}{}

\theoremstyle{plain}
\newtheorem{theorem}{Theorem}
\newtheorem{lemma}[theorem]{Lemma}
\newtheorem{corollary}[theorem]{Corollary}

\begin{document}

\title{Rooted $K_4$-Minors}

\author{Ruy Fabila-Monroy} \address{\newline Departamento de
  Matem\'aticas \newline Cinvestav \newline Distrito Federal,
  M\'exico} \email{ruyfabila@math.cinvestav.edu.mx}

\author{David~R.~Wood} \address{\newline Department of Mathematics and
  Statistics \newline The University of Melbourne \newline Melbourne,
  Australia} \email{woodd@unimelb.edu.au}

\thanks{R.F.-M. is supported by an Endeavour Fellowship from the
  Department of Education, Employment and Workplace Relations of the
  Australian Government. D.W. is supported by a QEII Research
  Fellowship from the Australian Research Council.}


\date{\today}

\begin{abstract}
  Let $a,b,c,d$ be four vertices in a graph $G$. A \emph{$K_4$-minor
    rooted} at $a,b,c,d$ consists of four pairwise-disjoint
  pairwise-adjacent connected subgraphs of $G$, respectively containing
  $a,b,c,d$. We characterise precisely when $G$ contains a $K_4$-minor
  rooted at $a,b,c,d$ by describing six classes of obstructions, which
  are the edge-maximal graphs containing no $K_4$-minor rooted at
  $a,b,c,d$.  The following two special cases illustrate the full
  characterisation: (1) A 4-connected non-planar graph contains a
  $K_4$-minor rooted at $a,b,c,d$ for every choice of $a,b,c,d$.  (2)
  A 3-connected planar graph contains a $K_4$-minor rooted at
  $a,b,c,d$ if and only if $a,b,c,d$ are not on a single face.
\end{abstract}

\maketitle

\section{Introduction}

Let $G$ and $H$ be graphs\footnote{We consider finite, simple,
  undirected graphs.}. An \emph{$H$-minor}\footnote{This definition of
  minor is a more concrete version of the standard definition: $H$ is a
  \emph{minor} of $G$ if $H$ is isomorphic to a graph obtained from a
  subgraph of $G$ by contracting edges.}  in $G$ is a set $\{G_x:x\in
V(H)\}$ of pairwise disjoint connected subgraphs of $G$ indexed by the
vertices of $H$, such that if $xy\in E(H)$ then some vertex in $G_x$
is adjacent to some vertex in $G_y$. Each subgraph $G_x$ is called a
\emph{branch set} of the minor. A complete graph $K_t$-minor in $G$ is
\emph{rooted} at distinct vertices $v_1,\dots,v_t\in V(G)$ if
$v_1,\dots,v_t$ are in distinct branch sets. For brevity, we say that a
$K_t$-minor rooted at $\{v_1,\dots,v_t\}$ is a
\emph{$\{v_1,\dots,v_t\}$-minor}. Rooted minors are a significant tool in
Robertson and Seymour's graph minor theory \citep{RS-GraphMinors}, and
a number of recent papers have studied rooted minors in their own
right \citep{JorgKawa-JGT07,Kawa-DM04,Wollan-JGT08,LinusWood}.  Rooted
minors are analogous to $H$-linked graphs for subdivisions; see
\citep{KostYu-DAM08,GKY-SJDM06,KostYu-JGT05}. This paper considers the
question:

\medskip{\it When does a given graph contain a $K_4$-minor rooted at
  four nominated vertices?}

\medskip Theorem~\ref{thm:Main} answers this
question by describing six classes of obstructions, which are the
edge-maximal graphs containing no $K_4$-minor rooted at four nominated
vertices. The flavour of this result is best introduced by first
considering the 3- and 4-connected cases, which are addressed in
Sections~\ref{sec:4Conn} and \ref{sec:3Conn}. First, we survey some
definitions and results from the literature that will be employed
later in the paper.

\section{Background}

The question of when does a graph contain a $K_3$-minor rooted at
three nominated vertices was answered by \citet{LinusWood}.

\begin{lemma}[\citep{LinusWood}]
  \label{lem:K3}
  For distinct vertices $a,b,c$ in a graph $G$, either:
  \begin{itemize}
  \item $G$ contains an $\{a,b,c\}$-minor, or
  \item for some vertex $v\in V(G)$ at most one of $a,b,c$ are in each
    component of $G-v$.
  \end{itemize}
\end{lemma}

Note that in this lemma it is possible that $v\in\{a,b,c\}$.

For distinct vertices $s_1,t_1,s_2,t_2$ in a graph $G$, an
\emph{$(s_1t_1,s_2t_2)$-linkage} consists of an $s_1t_1$-path and an
$s_2t_2$-path that are disjoint.
\citet{Seymour-DM80} and \citet{Thomassen-EuJC80} independently proved
that there is essentially one obstruction for the existence of a
linkage, as we now describe; see
\citep{Jung70,Watkins-DMJ68,Shiloach80,KLR,Hagerup,Woeginger,PerlShil,Tholey06}
for related results.

For a graph $H$, let $H^+$ denote a graph obtained from $H$ as
follows: for each triangle $T$ of $H$, add a possibly empty clique
$X_T$ disjoint from $H$ and adjacent to each vertex in $T$. We
consider $H^+$ to be implicitly defined by the graph $H$ and the
cliques $X_T$. An \emph{$(a,b,c,d)$-web} is a graph $H^+$, where $H$
is an embedded planar graph with outerface $(a,b,c,d)$, such that each
internal face of $H$ is a triangle, and each triangle of $H$ is a
face.
An \emph{$\{a,b,c,d\}$-web} is an $(a,b,c,d)$-web for some linear
ordering $(a,b,c,d)$.  That is, in an $\{a,b,c,d\}$-web the vertex
ordering around the outerface is not specified.

\begin{lemma}[\citep{Thomassen-EuJC80,Seymour-DM80}]
  \label{lem:linkweb}
  For distinct vertices $s_1,t_1,s_2,t_2$ in a graph $G$, either:
  \begin{itemize}
  \item $G$ contains an $(s_1t_1,s_2t_2)$-linkage, or
  \item $G$ is a spanning subgraph of an $(s_1,s_2,t_1,t_2)$-web.
  \end{itemize}
\end{lemma}

Lemma~\ref{lem:linkweb} implies the following result, first proved by
\citet{Jung70}.

\begin{lemma}[\citep{Jung70}]
  \label{lem:thom}
  For distinct vertices $s_1,s_2,t_1,t_2$ in a $4$-connected graph
  $G$, either:
  \begin{itemize}
  \item $G$ contains an $(s_1t_1,s_2t_2)$-linkage, or
  \item $G$ is planar and $s_1,s_2,t_1,t_2$ are on some face in this
    order.
  \end{itemize}
\end{lemma}

Lemma~\ref{lem:thom} makes sense since every 3-connected planar graph
has a unique planar embedding up to the choice of outerface
\citep{Whitney-AJM33c}. We implicitly use this fact throughout the
paper.

We now describe our first obstruction for a graph to contain a rooted
$K_4$-minor.

\begin{lemma}\label{lem:noweb}
  Every $(a,b,c,d)$-web $G$ contains no $\{a,b,c,d\}$-minor.
\end{lemma}

\begin{proof}[First proof]
  Since $G$ is an $(a,b,c,d)$-web, $G$ contains no $(ac,bd)$-linkage
  \citep{Seymour-DM80,Thomassen-EuJC80}.  But if $G$ contains a
  $K_4$-minor $A,B,C,D$ respectively rooted at $a,b,c,d$, then some
  $ac$-path (contained in $A\cup C$) is disjoint from some $bd$-path
  (contained in $B\cup D$). Thus $G$ contains no $\{a,b,c,d\}$-minor.
\end{proof}

\begin{proof}[Second proof]
  Suppose $G$ contains an $\{a,b,c,d\}$-minor. Since $G$ is connected,
  we may assume that every vertex is in some branch set.  Contracting
  each edge with both endpoints in the same branch set produces an
  outerplanar $K_4$, which is a contradiction.
\end{proof}

We will need the following result by \citet{Dirac60}.

\begin{lemma}[\citep{Dirac60}]
  \label{lem:dirac}
  For every set $S$ of $k$ vertices in a $k$-connected graph $G$,
  there is a cycle in $G$ containing $S$.
\end{lemma}

\section{The 4-Connected Case}
\label{sec:4Conn}

The following result characterises when a 4-connected graph contains a
rooted $K_4$-minor.  It is analogous to Lemma~\ref{lem:thom}.

\begin{theorem} \label{thm:4con} For distinct vertices $a,b,c,d$ in a
  $4$-connected graph $G$, either:
  \begin{itemize}
  \item $G$ contains an $\{a,b,c,d\}$-minor, or
  \item $G$ is planar and $a,b,c,d$ are on a common face.
  \end{itemize}
\end{theorem}

\begin{proof}
  Lemma~\ref{lem:noweb} implies that if $G$ contains an
  $\{a,b,c,d\}$-minor, then the second outcome does not occur.  To
  prove the converse, assume that $G$ is non-planar, or if $G$ is
  planar then $a,b,c,d$ are not on a common face. Since $G$ is
  4-connected, by Lemma~\ref{lem:dirac}, $G$ contains a cycle $C$
  through $a,b,c,d$.  Without loss of generality, $a,b,c,d$ appear in
  this order in $C$.  By Lemma~\ref{lem:thom}, $G$ contains an
  $(ac,bd)$-linkage. The result follows from Lemma~\ref{lem:cycle}
  below.
\end{proof}

\begin{lemma} \label{lem:cycle} Let $C$ be a cycle in a graph $G$
  containing vertices $a,b,c,d$ in this order. If $G$ contains an
  $(ac,bd)$-linkage then $G$ contains an $\{a,b,c,d\}$-minor.
\end{lemma}

\begin{proof}
  Let $G$ be a counterexample firstly with $|V(G)|$ minimum and then
  with $|E(G)|$ minimum. If $V(G)=\{a,b,c,d\}$ then $G\cong K_4$. Now
  assume that $|V(G)|\geq 5$, and the result holds for graphs with
  less than $|V(G)|$ vertices, or with $|V(G)|$ vertices and less than
  $|E(G)|$ edges. 

  Let $P$ be an $ac$-path disjoint from some $bd$-path $Q$.  Let
  $R_{ab}$ be the $ab$-path contained in $C$ avoiding $c$ and $d$.
  Similarly define $R_{bc}$, $R_{cd}$ and $R_{da}$. If some vertex or
  edge $x$ is not in $P\cup Q\cup C$, then $G-x$ is not a
  counterexample, and thus contains an $\{a,b,c,d\}$-minor. Now assume
  that $G=P\cup Q\cup C$. We show that contracting some edge gives a
  graph that satisfies the hypothesis.

  Suppose that some vertex $v$ has degree 2.  For at least one edge
  $e$  incident to $v$, the endpoints  of $e$ are not both in
  $\{a,b,c,d\}$. Thus the   contraction $G/e$ satisfies the
  hypothesis, and  $G/e$ and hence $G$
  contains an $\{a,b,c,d\}$-minor. Now assume that every vertex has
  degree at least 3. Thus $V(G)=V(C)=V(P\cup Q)$.

  Colour $P$ red, and colour $Q$ blue. Suppose that consecutive
  vertices $u$ and $v$ in $C$ receive the same colour.  Then $G/uv$
  satisfies the hypothesis, as illustrated in
  Figure~\ref{fig:consecutive_cycle} in the case that $u$ and $v$ are
  red. By the choice of $G$, $G/uv$ and thus $G$ contains an
  $\{a,b,c,d\}$-minor. Now assume that the colours alternate around
  $C$. In particular, $|V(P)|=|V(Q)|$. If $P=ac$ then $Q=bd$ and and
  we are done. Now assume that $P$ contains some internal vertex.

  \begin{figure}[!htb]
    \begin{center}
      \includegraphics[height=45mm]{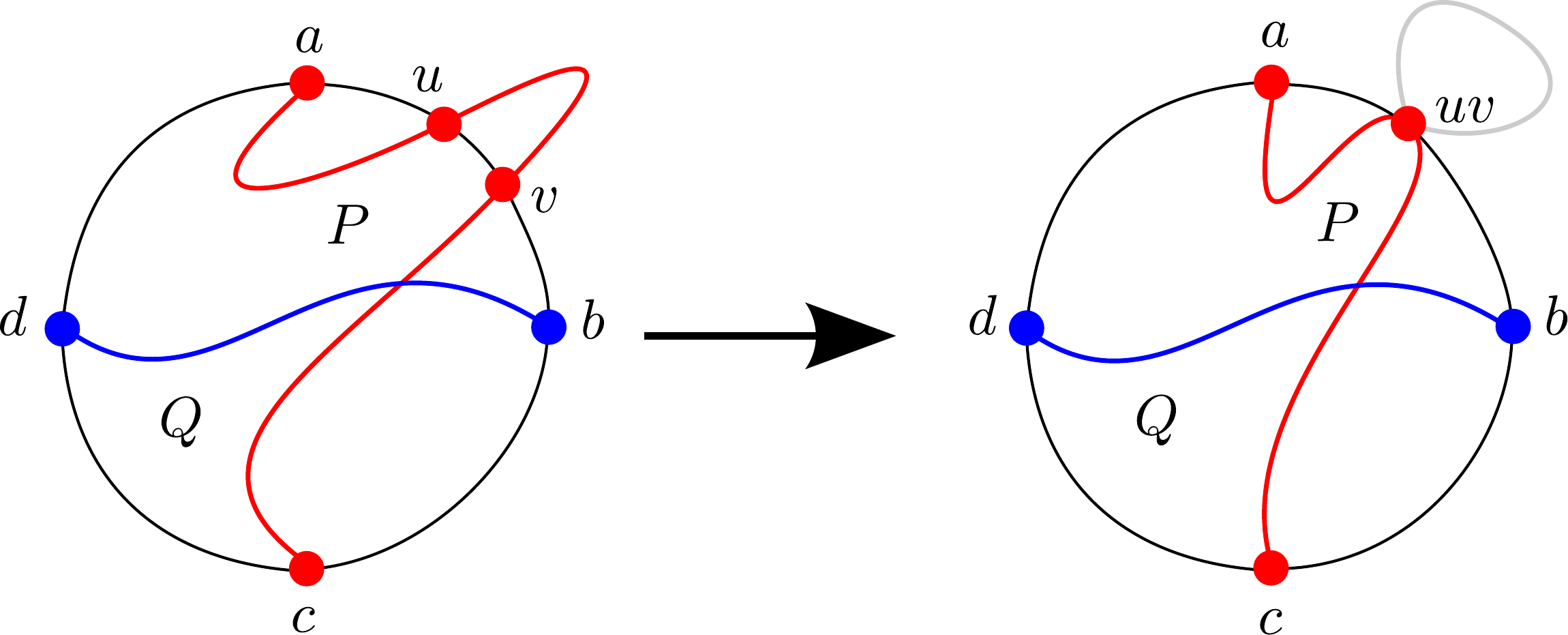}
    \end{center}
    \caption{\label{fig:consecutive_cycle} If consecutive vertices $u$
      and $v$ in $C$ receive the same colour then contract $uv$.}
  \end{figure}

  Let $v$ be the neighbour of $a$ in $P$, and let $w$ be the neighbour
  of $c$ in $P$. If $v$ is in $R_{da}\cup R_{ab}$, then $G/av$
  satisfies the hypothesis, as illustrated in
  Figure~\ref{fig:next_cycle}. By the choice of $G$, $G/av$ and thus $G$
  contains an $\{a,b,c,d\}$-minor. Now assume that $v\in R_{bc}\cup
  R_{cd}$. Similarly, $w\in R_{da}\cup R_{ab}$. Since $P$ and $Q$ are
  disjoint, $v\in R_{bc}\cup R_{cd}\setminus \{b,d\}$ and $w\in
  R_{da}\cup R_{ab}\setminus\{b,d\}$. Thus $v\neq w$.  That is, $P$
  (and $Q$ also) contains at least two internal vertices.  Label $v$
  and $a$ by ``$a$''.  Label every other vertex in $P$ by ``$c$''.

  \begin{figure}[H]
    \begin{center}
      \includegraphics[height=45mm]{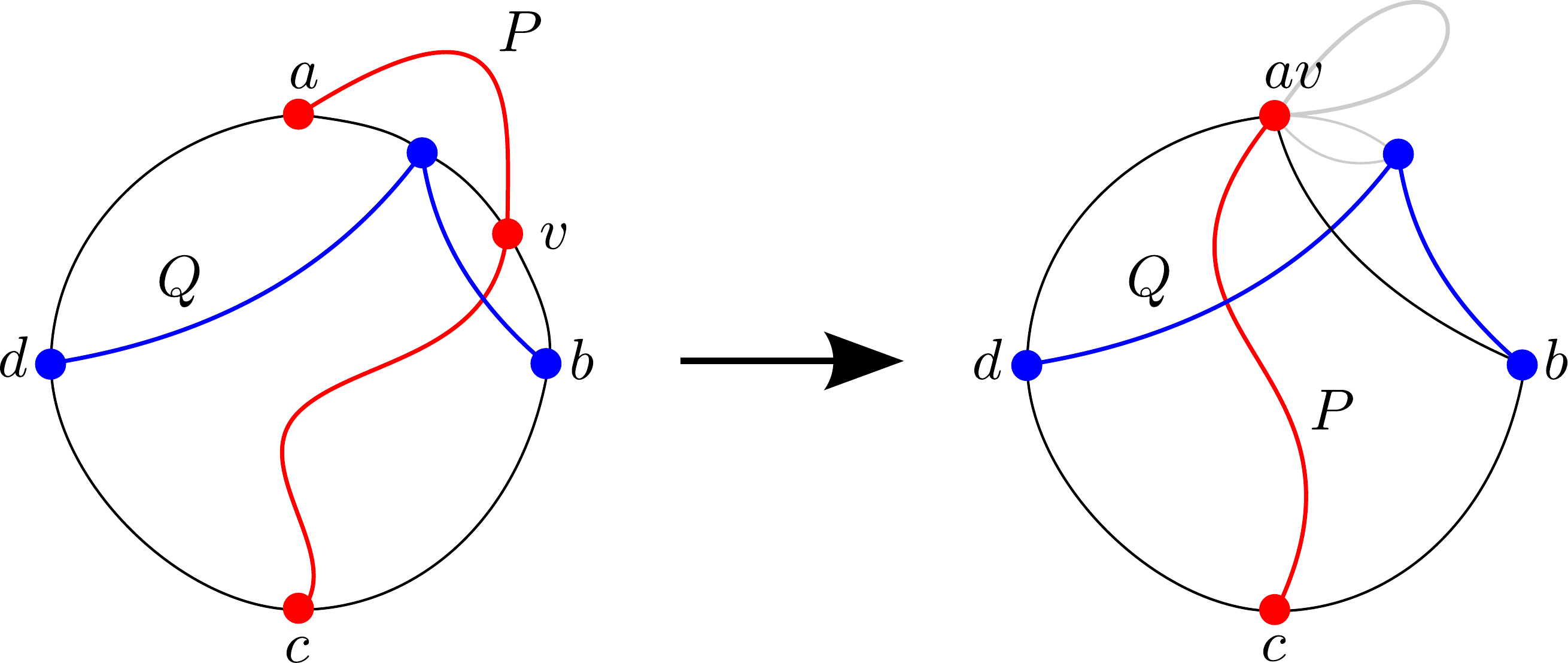}
    \end{center}
    \caption{\label{fig:next_cycle} If $v$ is in $R_{da}\cup R_{ab}$
      then contract $av$.}
  \end{figure}

  Let $x$ be the neighbour of $v$ between $v$ and $c$ in $R_{bc}\cup
  R_{cd}$.  Let $y$ be the neighbour of $a$ between $w$ and $a$ in
  $R_{da}\cup R_{ab}$. Since the colours around $C$ alternate, $x$ and
  $y$ are in $Q$. Without loss of generality, $b,x,y,d$ appear in this
  order in $Q$.  Label the $yd$-subpath of $Q$ by ``$d$'', and label
  the remaining vertices in $Q$ (including $x$) by ``$b$''.  Thus $x$,
  which is labelled ``$b$'', is adjacent to some vertex in $Q$
  labelled ``$d$''. The neighbours of $x$ in $C$ are labelled ``$a$''
  and ``$c$'', and the neighbours of $y$ in $C$ are labelled ``$a$''
  and ``$c$''.  The sets of vertices labelled
  ``$a$'',``$b$'',``$c$'',``$d$'' form pairwise disjoint subpaths of
  $P$ or $Q$ respectively containing $a,b,c,d$.  Thus contracting the
  vertices with the same label into a single vertex gives an
  $\{a,b,c,d\}$-minor in $G$, as illustrated in Figure~\ref{fig:K4minor}.
\end{proof}

  \begin{figure}[H]
    \begin{center}
      \includegraphics[height=45mm]{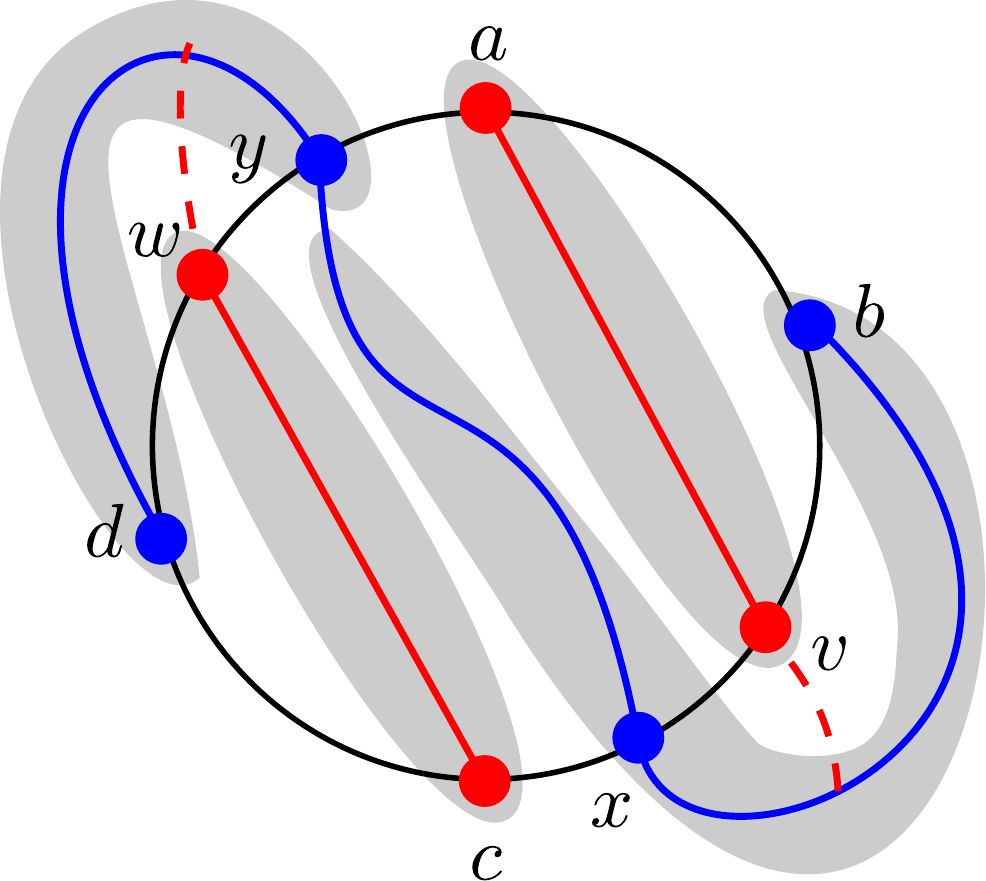}
    \end{center}
    \caption{\label{fig:K4minor} Construction of a rooted $K_4$-minor
      in Lemma~\ref{lem:cycle}. }
  \end{figure}

  \section{The 3-Connected Case}
  \label{sec:3Conn}

  We have the following characterisation for $3$-connected graphs.

  \begin{theorem} \label{thm:3con} The following are equivalent for
    distinct vertices $a,b,c,d$ in a $3$-connected graph $G$:
    \begin{enumerate}
    \item $G$ contains an $\{a,b,c,d\}$-minor,
    \item $G$ is not a spanning subgraph of an $\{a,b,c,d\}$-web,
    \item $G$ contains an $(ab,cd)$-linkage, an $(ac,bd)$-linkage, and
      an $(ad,bc)$-linkage.
    \end{enumerate}
  \end{theorem}

\begin{proof}
  Lemma~\ref{lem:noweb} implies (1) $\Longrightarrow$ (2).
  Lemma~\ref{lem:linkweb} implies (2) $\Longrightarrow$ (3).  It
  remains to prove (3) $\Longrightarrow$ (1).  First suppose that
  some cycle $C$ contains $a,b,c,d$.  Without loss of generality
  assume that the order of the vertices in $C$ is $(a,b,c,d)$. Since
  $G$ contains an $(ac,bd)$-linkage, by Lemma \ref{lem:cycle}, $G$
  contains an $\{a,b,c,d\}$-minor. Now assume that no cycle contains
  $a,b,c,d$. By Lemma~\ref{lem:dirac}, since $G$ is $3$-connected, $G$
  contains a cycle $C$ through $a,b,c$. Colour red the vertices in the
  $ab$-path in $C$ that avoids $c$. Likewise colour blue the vertices
  in the $bc$-path in $C$ that avoids $a$. And colour green the
  vertices in the $ca$-path in $C$ that avoids $b$. Note that $a,b$
  and $c$ each receive two colours.  By Menger's Theorem there exists
  three paths from $d$ to $C$, such that each path intersects $C$ in
  one vertex, and any two of the paths only intersect at $d$. Colour
  each path with the colour of its vertex in $C$. If two paths receive
  the same colour, then we obtain a cycle through $a,b,c,d$, as
  illustrated in Figure~\ref{fig:3con}(a).  Now assume that no two
  paths receive the same colour. In this case we obtain an
  $\{a,b,c,d\}$-minor, as illustrated in Figure~\ref{fig:3con}(b).
\end{proof}

\begin{figure}[!htb]
  \begin{center}
    \includegraphics[width=0.8\textwidth]{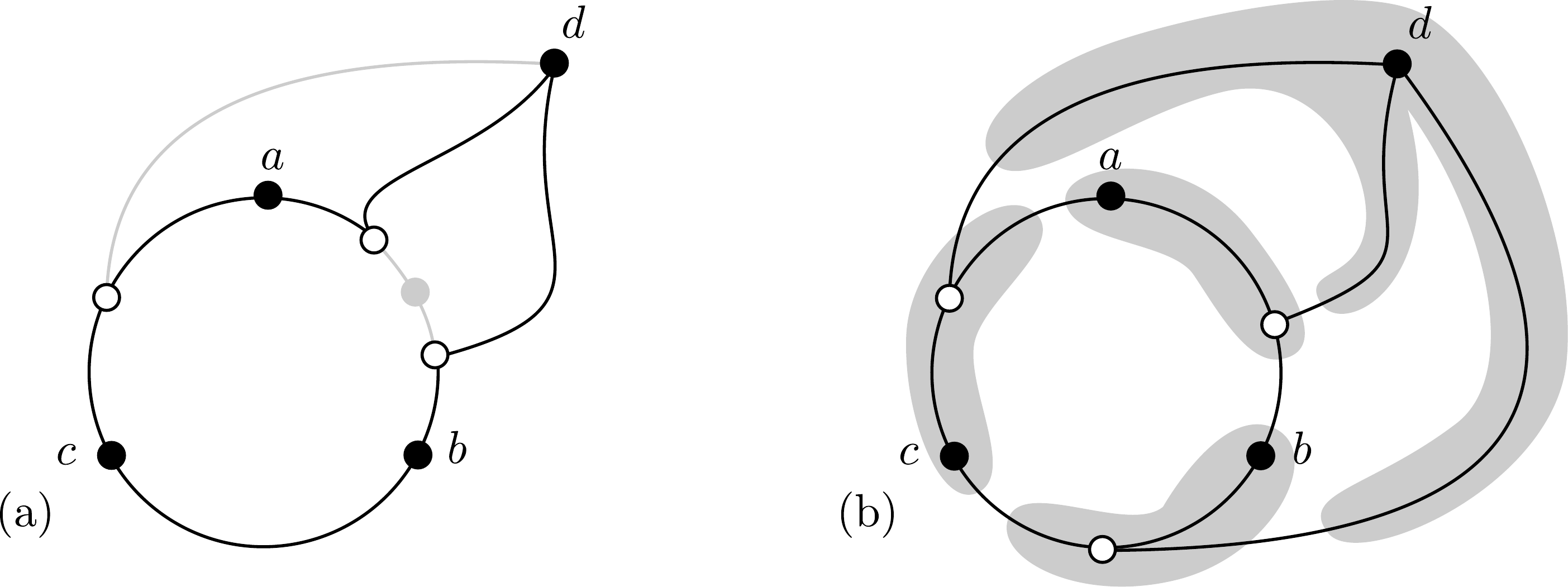}
  \end{center}
  \caption{\label{fig:3con} Finding a rooted $K_4$-minor in a
    3-connected graph.}
\end{figure}

Note that Theorem~\ref{thm:3con} does not hold for 2-connected
graphs. For example, $K_{2,3}$ with colour classes $\{a,b,c\}$ and
$\{d,v\}$ contains an $(ab,cd)$-linkage, an $(ac,bd)$-linkage, and an
$(ad,bc)$-linkage, but contains no $\{a,b,c,d\}$-minor.

Theorem~\ref{thm:3con} can be strengthened for 3-connected planar
graphs.

\begin{theorem}
  \label{thm:3ConPlanar}
  For distinct vertices $a,b,c,d$ in a 3-connected planar graph $G$,
  either:
  \begin{itemize}
  \item $G$ contains an $\{a,b,c,d\}$-minor, or
  \item $a,b,c,d$ are on a common face.
  \end{itemize}
\end{theorem}

\begin{proof}
  If $a,b,c,d$ are on a common face, then $G$ is a spanning subgraph
  of an $\{a,b,c,d\}$-web; thus $G$ contains no $\{a,b,c,d\}$-minor by
  Lemma~\ref{lem:noweb}.  For the converse, assume that $G$ contains no
  $\{a,b,c,d\}$-minor. By Theorem~\ref{thm:3con}, $G$ is a spanning
  subgraph of $H^+$ for some planar graph $H$ with outerface
  $\{a,b,c,d\}$, such that every internal face of $H$ is a triangle.
  Suppose that for some triangular face $T=(u,v,w)$ of $H$, at least
  two vertices $x,y\in X_T$ are adjacent in $G$ to each of
  $u,v,w$. Let $z$ be a vertex of $H$ outside of $T$. There is such a
  vertex since the outerface has four vertices. Since $G$ is
  3-connected, there are three internally disjoint $xz$-paths,
  respectively passing through $u,v,w$. Thus $G$ contains a
  subdivision of $K_{3,3}$ with colours classes $\{u,v,w\}$ and
  $\{x,y,z\}$. This contradiction proves that for each triangular face
  $T=(u,v,w)$ of $H$, at most one vertex in $X_T$ is adjacent to each
  of $u,v,w$ in $G$. If there is such a vertex $x\in X_T$ then move
  $x$ into $H$. Observe that $H$ remains planar: the face $uvw$ is
  replaced by the faces $T_w=(u,v,x)$, $T_v=(u,w,x)$ and
  $T_u=(v,w,x)$. Each remaining vertex in $X_T$ is now adjacent to at
  most two of $u,v,w$ (and possibly $x$). Assign such a vertex to one
  of $X_{T_u},X_{T_v},X_{T_w}$ according to its neighbours in
  $T$. Repeat this step until $X_T=\emptyset$ for each triangle $T$ of
  $H$.  In this case, $G$ is a spanning subgraph of $H$ (not $H^+$),
  and $a,b,c,d$ are on a common face of $G$.
\end{proof}

\begin{corollary}
  A planar triangulation contains an $\{a,b,c,d\}$-minor for all
  distinct vertices $a,b,c,d$.
\end{corollary}

\section{Reductions}

This section describes a number of operations that simplify the search
for rooted $K_4$-minors. The first motivates the definition of $H^+$.

\begin{lemma}
  \label{lem:KillCliques}
  Let $a,b,c,d$ be distinct vertices in a graph $H$. For each graph
  $H^+$, we have $H^+$ contains an $\{a,b,c,d\}$-minor if and only if
  $H$ contains an $\{a,b,c,d\}$-minor.
\end{lemma}

\begin{proof}
  Since $H$ is a subgraph of $H^+$, if $H$ contains an
  $\{a,b,c,d\}$-minor then so does $H^+$. For the converse, say
  $A,B,C,D$ is a $K_4$-minor in $H^+$ rooted at $a,b,c,d$. Let $A':=
  A\cap H$. Define $B',C',D'$ similarly. Suppose that $A'$ intersects
  the clique $X_T$ associated with some triangle $T$ of $H$.  Since
  $T$ separates $a$ and $X_T$, $A'$ intersects $T$. Since the vertices
  in $A\cap T$ are pairwise adjacent, $A\cap H$ is a connected
  subgraph of $H$. If two branch sets, say $A$ and $B$, are adjacent
  in $X_T$, then they both contain a vertex in $T$, and $A'$ and $B'$
  are adjacent in $H$. Thus $A',B',C',D'$ is a $K_4$-minor in $H$
  rooted at $a,b,c,d$.
\end{proof}

A \emph{separation} of a graph $G$ is an ordered pair $(G_1,G_2)$ of
subgraphs of $G$ such that $G=G_1 \bigcup G_2$, and
$G_1\not\subseteq G_2$ and $G_2\not\subseteq G_1$. So there
is no edge between $G_1-G_2$ and $G_2-G_1$. The \emph{order} of
$(G_1,G_2)$ is $|V(G_1\cap G_2)|$. If certain vertices in $G$ are
nominated, and there are $s$ nominated vertices in $G_1$ and $t$
nominated vertices in $G_2$, then $(G_1,G_2)$ is an
\emph{$(s,t)$-separation}.

\begin{lemma}
  \label{lem:22}
  Let $a,b,c,d$ be four nominated vertices in a 2-connected graph $G$.
  Let $(G_1,G_2)$ be a $(2,2)$-separation of $G$ of order 2, such that
  $a,b \in V(G_1)$ and $c,d \in V(G_2)$. Let $\{u,v\}:=V(G_1)\cap
  V(G_2)$. Let $G_i'$ be the graph obtained from $G_i$ by adding the
  edge $uv$. Then $G$ contains an $\{a,b,c,d\}$-minor if and only if
  $G_1'$ contains an $\{a,b,u,v\}$-minor or $G_2'$ contains a
  $\{u,v,c,d\}$-minor.
\end{lemma}

\begin{proof}
  Since $G$ is 2-connected, $G_2'$ can obtained from $G$ by
  contracting $G_1$ onto the edge $uv$, and $G_1'$ can obtained from
  $G$ by contracting $G_2$ onto $uv$. Thus, if $G_1'$ contains an
  $\{a,b,u,v\}$-minor or $G_2'$ contains a $\{u,v,c,d\}$-minor, then
  $G$ contains an $\{a,b,c,d\}$-minor. For the converse, assume that
  $G$ contains a $K_4$-minor $A,B,C,D$ containing $a,b,c,d$
  respectively. Grow the branch sets until $u$ and $v$ are in $A\cup
  B\cup C\cup D$. Without loss of generality, $u$ is in $A$. Thus $v$
  separates $b$ from $\{c,d\}$ in $G-A$. Hence $v$ is in
  $B$. Therefore $A\cap G_2,B\cap G_2,C,D$ is a $\{u,v,c,d\}$-minor of
  $G_2$.
\end{proof}

\begin{lemma}
  \label{lem:ear}
  Let $G$ be a graph with four nominated vertices $a,b,c,d$, such that
  $N_G(a)=N_G(b)=\{u,v\}$ for some vertices $u,v\in
  V(G)\setminus\{a,b,c,d\}$. Let $G'$ be the graph obtained from $G$
  by deleting $a$ and $b$, and adding the edge $uv$. Then $G$ contains
  an $\{a,b,c,d\}$-minor if and only if $G'$ contains a
  $\{u,v,c,d\}$-minor.
\end{lemma}

\begin{proof}
  If $G'$ contains a $\{u,v,c,d\}$-minor, then contracting the edges
  $au$ and $bv$ gives an $\{a,b,c,d\}$-minor in $G$. For the converse,
  say $A,B,C,D$ is a $K_4$-minor in $G$ respectively rooted at
  $a,b,c,d$. Grow the branch sets until $u$ and $v$ are in $A\cup
  B\cup C\cup D$. If $u$ is in $C$ then $v$ separates $\{a,b\}$ and
  $D$, implying $v$ is in $D$, in which case $A=\{a\}$ and $B=\{b\}$,
  and $A$ and $B$ are not adjacent. By symmetry, $\{u,v\}\cap(C\cup
  D)=\emptyset$. Thus $u,v\in A\cup B$. If $u,v\in A$ then $A$
  separates $b$ and $C\cup D$. Thus $u\in A$ and $v\in B$, without
  loss of generality. Hence $A-a,B-b,C,D$ is a $\{u,v,c,d\}$-minor in
  $G'$.
\end{proof}

\section{Obstructions}

Consider the following classes of graphs, each of which contains no
$K_4$-minor rooted at the four nominated vertices. Each graph in each
class is called an \emph{obstruction}; see
Figure~\ref{fig:obstructions}.

\begin{description}
\item[\normalfont Class $\mathcal{A}$] Let $H$ be the graph consisting
  of an edge $pq$ with $p$ nominated, and three nominated vertices
  adjacent to both $p$ and $q$. Let $\mathcal{A}$ be the class of all
  graphs $H^+$.

\item[\normalfont Class $\mathcal{B}$] Let $H$ be the graph consisting
  of an edge $pq$, and four nominated vertices adjacent to both $p$
  and $q$. Let $\mathcal{B}$ be the class of all graphs $H^+$.

\item[\normalfont Class $\mathcal{C}$] Let $H$ be the graph consisting
  of a triangle $uvw$, plus two nominated vertices adjacent to $u$ and
  $v$, and two nominated vertices adjacent to $v$ and $w$.  Let
  $\mathcal{C}$ be the class of all graphs $H^+$.

\item[\normalfont Class $\mathcal{D}$] Let $H$ be a planar graph with
  an outerface of four nominated vertices, such that every internal
  face is a triangle, and every triangle is a face.  Let $\mathcal{D}$
  be the class of all graphs $H^+$. (These are the webs.)

\item[\normalfont Class $\mathcal{E}$] Let $H$ be a planar graph with
  outerface $(p,q,r,s)$ where $p$ and $q$ are nominated, every
  internal face is a triangle, and every triangle is a face. Add
  to $H$   two nominated vertices $v$ and $w$ adjacent to $r$ and $s$.  Let
  $\mathcal{E}$ be the class of all graphs $H^+$.

\item[\normalfont Class $\mathcal{F}$] Let $H$ be a planar graph with
  outerface $(p,q,r,s)$ where every other face is a triangle and every
  triangle is a face. Add to $H$ two nominated vertices
  adjacent to $p$ and $q$, and two nominated vertices
  adjacent to $r$ and $s$.  Let $\mathcal{F}$ be the class of all
  graphs $H^+$.

\end{description}

\begin{figure}
  \begin{center}
    \includegraphics[width=0.8\textwidth]{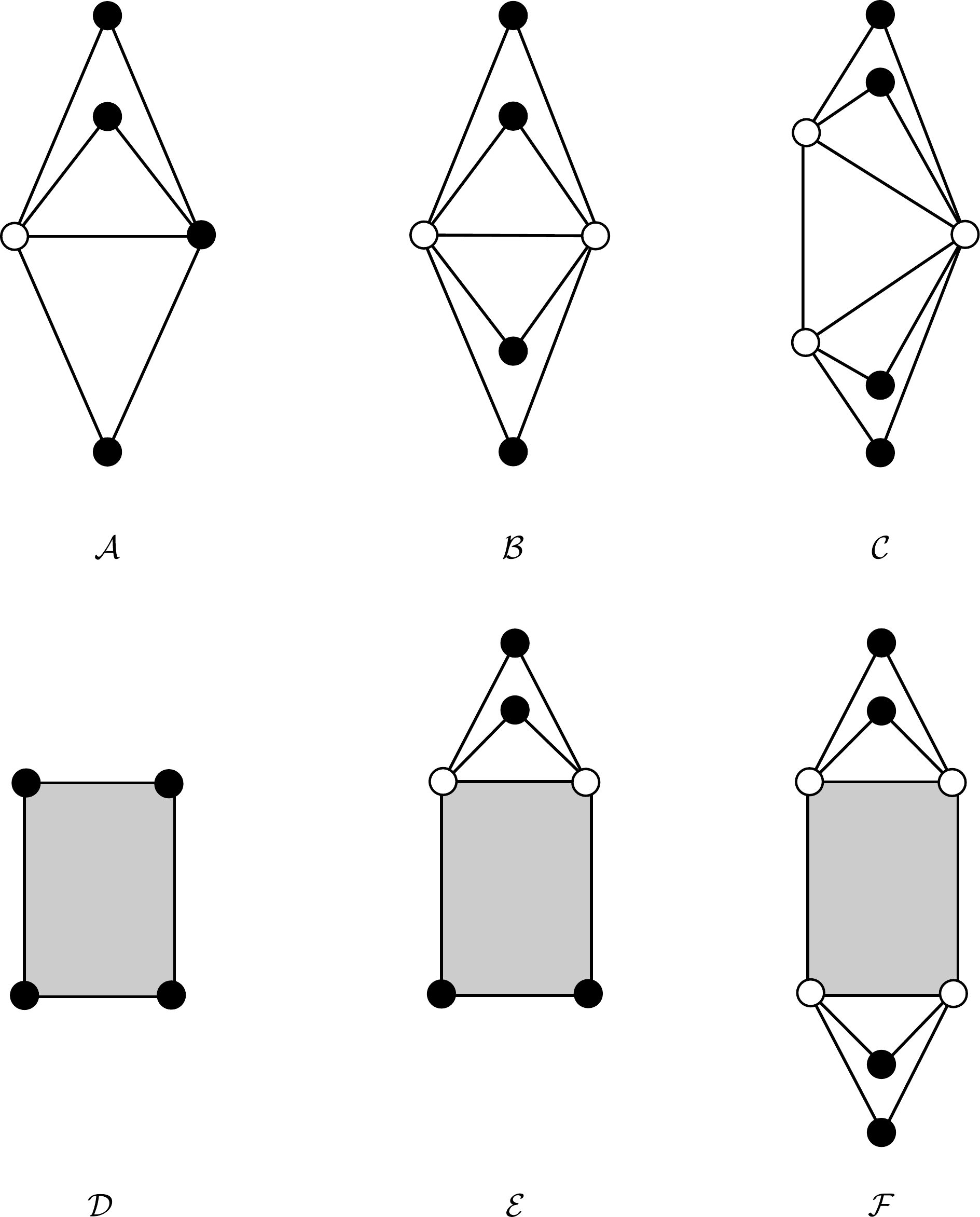}
  \end{center}
  \caption{\label{fig:obstructions} The obstructions. Nominated
    vertices are dark. Non-nominated vertices are white. Shaded
    regions represent a web. Adjacent to each triangle is an undrawn
    clique. }
\end{figure}

The \emph{type} of a nominated vertex $x$ in one of the above
obstructions $H^+$ is defined as follows:
\begin{description}
\item[\normalfont Type-1] $H^+\in \mathcal{D}\cup\mathcal{E}$, and $x$
  is adjacent to some other nominated vertex in $H$.

\item[\normalfont Type-2] $H^+\in\mathcal{A}$, and $x$ has
  degree 4 in $H$.

\item[\normalfont Type-3] $H^+\in
  \mathcal{A}\cup\mathcal{B}\cup\mathcal{C}\cup\mathcal{E}\cup\mathcal{F}$,
  and $x$ is neither type-1 nor type-2; such a vertex $x$ has degree 2 in
  $H$,

\end{description}

\begin{lemma}
  \label{lem:obstructions}
  Every graph in
  $\mathcal{A}\cup\mathcal{B}\cup\mathcal{C}\cup\mathcal{D}\cup\mathcal{E}\cup\mathcal{F}$
  contains no $K_4$-minor rooted at the four nominated vertices.
\end{lemma}

\begin{proof}
Lemma~\ref{lem:noweb} implies the result for a class $\mathcal{D}$
obstruction. Let $H^+$ be an
  obstruction in some other class. By Lemma~\ref{lem:KillCliques}, it
  suffices to prove that $H$ contains no $\{a,b,c,d\}$-minor, where
  $a,b,c,d$ are the four nominated vertices.

  If $H^+\in\mathcal{A}$ then $H\cong K_{1,1,3}$, in which case
  contracting an edge incident to the one non-nominated vertex produces
  $K_4-e$ or $K_{1,3}$, neither of which are $K_4$.

  For
  $H^+\in\mathcal{B}\cup\mathcal{C}\cup\mathcal{E}\cup\mathcal{F}$,
  Lemma~\ref{lem:ear} is applicable. In particular,
  $N_H(a)=N_H(b)=\{u,v\}$ for some vertices $u,v\in
  V(H)\setminus\{a,b,c,d\}$. Thus if $H'$ is the graph obtained from
  $H$ by deleting $a$ and $b$, and adding the edge $uv$, then $H^+$
  contains an $\{a,b,c,d\}$-minor if and only if $H$ contains an
  $\{a,b,c,d\}$-minor if and only if $H'$ contains a
  $\{u,v,c,d\}$-minor.

  If $H^+\in\mathcal{B}$ then $H'\cong K_4-e$.  Thus in each case,
  $H'$ contains no $\{u,v,c,d\}$-minor, implying that $H$ contains no
  $\{a,b,c,d\}$-minor.  If $H^+\in \mathcal{C}$ then
  $H'\in\mathcal{A}$, which has no $\{u,v,c,d\}$-minor as proved
  above.  If $H^+\in \mathcal{E}$ then $H'\in\mathcal{D}$, which has
  no $\{u,v,c,d\}$-minor by Lemma~\ref{lem:noweb}.  If $H^+\in
  \mathcal{F}$ then $H'\in\mathcal{E}$, which has no
  $\{u,v,c,d\}$-minor as proved above.
\end{proof}

\section{Main Theorem}
\label{sec:Main}

We now state and prove the main result of the paper.  It characterises
when a given graph contains a $K_4$-minor rooted at four nominated
vertices.

\begin{theorem}
  \label{thm:Main}
  For every graph $G$ with four nominated vertices, either:
  \begin{itemize}
  \item $G$ contains a $K_4$-minor rooted at the nominated vertices,
    or
  \item $G$ is a spanning subgraph of a graph in
    $\mathcal{A}\cup\mathcal{B}\cup\mathcal{C}\cup\mathcal{D}\cup\mathcal{E}\cup\mathcal{F}$
  \end{itemize}
\end{theorem}

\begin{proof}
  Lemma~\ref{lem:obstructions} proves that both outcomes are not
  simultaneously possible. Suppose on the contrary that for some graph
  $G$ neither outcome occurs. That is, $G$ contains no $K_4$-minor
  rooted at the nominated vertices, and $G$ is not a spanning subgraph
  of a graph in
  $\mathcal{A}\cup\mathcal{B}\cup\mathcal{C}\cup\mathcal{D}\cup\mathcal{E}\cup\mathcal{F}$. Choose
  $G$ firstly with $|V(G)|$ minimum, and then with $|E(G)|$
  maximum. Let $a,b,c,d$ be the nominated vertices in $G$.  If
  $|V(G)|=4$ then $G$ contains an $\{a,b,c,d\}$-minor if and only if
  $G\cong K_4$. Otherwise, $G$ is a subgraph of $K_4$ minus an edge,
  which is in class $\mathcal{D}$. Now assume that $|V(G)|\geq5$ and
  the result holds for every graph $G'$ with $|V(G')|<|V(G)|$, or
  $|V(G')|=|V(G)|$ and $|E(G')|>|E(G)|$. We proceed by considering the
  possible separations in $G$.

  \begin{itemize}
  \item Suppose there is a $(0,4)$-separation $(G_1,G_2)$ of order 0:
    If $G_2$ contains a $K_4$-minor rooted at the nominated vertices,
    then so does $G$. Otherwise, by the choice of $G$, $G_2$ is a
    spanning subgraph of an obstruction $H^+$. Adding $V(G_1)$ to
    $X_T$ for some triangle $T$ of $H$, we obtain an obstruction
    containing $G$ as a spanning subgraph, as desired.

  \item Suppose there is a $(1,3)$-separation $(G_1,G_2)$ of order 0:
    Let $a$ be the nominated vertex in $G_1$. Let $b,c,d$ be the
    nominated vertices in $G_2$. Thus $G$ contains no $ab$-path. Hence
    $G$ contains no $\{a,b,c,d\}$-minor. Let $H:=K_4-ad$ with
    $V(H):=\{a,b,c,d\}$. Let $X_{abc}:=V(G_1)\setminus\{a\}$ and
    $X_{bcd}:=V(G_2)\setminus\{b,c,d\}$. Hence $G$ is a spanning
    subgraph of $H^+$, a class $\mathcal{D}$ obstruction.

  \item Suppose there is a $(2,2)$-separation $(G_1,G_2)$ of order 0:
    Then as in the proof of the previous case, $G$ contains no
    $\{a,b,c,d\}$-minor and $G$ is a spanning subgraph of a class
    $\mathcal{D}$ obstruction.

  \end{itemize}

  Now assume that $G$ is connected.

  \begin{itemize}

  \item Suppose that $(G_1,G_2)$ is a $(0,4)$-separation of order $1$:
    Let $\{u\}:=V(G_1\cap G_2)$. If $G_2$ contains an
    $\{a,b,c,d\}$-minor then so does $G$, and we are done. Otherwise,
    by the choice of $G$, $G_2$ is a spanning subgraph of an
    obstruction $H^+$. Now, $u$ is in $T\cup X_T$ for some triangle
    $T$ of $H$. Add $V(G_1)\setminus\{u\}$ to $X_T$. The resulting
    graph $H^+$ is in the same class as the original $H^+$ and
    contains $G$ as a spanning subgraph.

  \item Suppose that $(G_1,G_2)$ is a $(1,3)$-separation of order $1$:
    Let $\{u\}:=V(G_1\cap G_2)$. Let $a$ be the nominated vertex in
    $G_1-G_2$. If $G_2$ contains an $\{u,b,c,d\}$-minor, then adding
    $G_1$ to the branch set that contains $u$ gives an
    $\{a,b,c,d\}$-minor in $G$, and we are done. Otherwise, by the
    choice of $G$, $G_2$ is a spanning subgraph of an obstruction
    $H^+$, where $u,b,c,d$ are nominated in $G_2$.
    
    If $u$ is type-1, then $u$ is in the outerface of $H$ (as embedded
    in Figure~\ref{fig:obstructions}).  Let $x$ and $y$ be the two
    neighbours of $u$ in this outerface.  Add $a$ into the outerface
    of $H$, adjacent to $x$, $u$ and $y$.  Thus $axu$ and $auy$ become
    internal faces of $H$. Let $X_{axu}:=V(G_1)\setminus \{a,u\}$. The
    resulting graph $H^+$ contains $G$ as a spanning subgraph, and is
    in the same class as the original $H^+$.

    If $u$ is type-2, then $H^+$ is in class $\mathcal{A}$. Let $x$ be
    the degree-4 neighbour of $u$ in $H$. Add $a$ to $H$ adjacent to
    $u$ and $x$, thus creating the triangle $axu$. Let
    $X_{axu}:=V(G_1)\setminus \{a,u\}$. The resulting graph $H^+$
    (with $a$ nominated) is in class $\mathcal{B}$, and contains $G$
    as a spanning subgraph.

    If $u$ is type-3, then $u$ is in a unique triangle $uxy$ in $H$.
    In $H$, delete $u$, add $a$ adjacent to $x$ and $y$, thus creating
    the triangle $axy$. Let $X_{axy}:= V(X_{uxy}) \cup
    V(G_1)\setminus\{a\}$. The resulting graph $H^+$ (with $a$
    nominated) is in the same class as the original $H^+$, and
    contains $G$ as a spanning subgraph.

  \item Suppose that $(G_1,G_2)$ is a $(2,2)$-separation of order $1$:
    Let $\{u\}:=V(G_1\cap G_2)$.  Without loss of generality, $a,b\in
    V(G_1)$ and $c,d\in V(G_2)$. Let $H$ be the planar graph with
    outerface $(a,b,c,d)$, and one internal vertex $u$ adjacent to
    $a,b,c,d$.  Let $X_{abu}:=V(G_1)\setminus\{a,b,u\}$ and
    $X_{cdu}:=V(G_2)\setminus\{c,d,u\}$. The resulting graph $H^+$ is
    in class $\mathcal{D}$, and contains $G$ as a spanning subgraph.

  \item Suppose that $(G_1,G_2)$ is a $(1,4)$-separation of order $1$:
    Without loss of generality, $a\in V(G_1)$ and $a,b,c,d\in
    V(G_2)$. If $G_2$ contains an $\{a,b,c,d\}$-minor then so does
    $G$. Otherwise, by the choice of $G$, $G_2$ is a spanning subgraph
    of an obstruction $H^+$. Now, $a$ is in some triangle $T$ of
    $H$. Add $V(G_1)\setminus\{a\}$ to $X_T$. The resulting graph
    $H^+$ is in the same class as the original $H^+$, and contains $G$
    as a spanning subgraph.

  \item Suppose that $(G_1,G_2)$ is a $(2,3)$-separation of order $1$:
    Without loss of generality, $a,b\in V(G_1)$ and $b,c,d\in
    V(G_2)$. Let $H:=K_4-ad$ where $V(H):=\{a,b,c,d\}$. Let
    $X_{abc}:=V(G_1)\setminus\{a,b\}$ and
    $X_{bcd}:=V(G_2)\setminus\{b,c,d\}$. The resulting graph $H^+$ is
    in class $\mathcal{D}$, and contains $G$ as a spanning subgraph.
  \end{itemize}

  Now assume that $G$ is 2-connected.

  \begin{itemize}
  \item Suppose there is a $(0,4)$-separation $(G_1,G_2)$ of order 2,
    or a $(1,4)$-separation $(G_1,G_2)$ of order 2, or a
    $(2,4)$-separation $(G_1,G_2)$ of order 2: Let $\{u,v\}:=V(G_1\cap
    G_2)$. Let $G'$ be the graph obtained by contracting $G_1$ onto
    the edge $uv$. (This is possible since $G$ is 2-connected.)\ If
    $G'$ contains an $\{a,b,c,d\}$-minor then so does $G$, and we are
    done. Otherwise, by the choice of $G$, $G'$ is a spanning subgraph
    of an obstruction $H^+$. Since $uv$ is an edge of $G'$, we have
    $u,v\in T\cup X_T$ for some triangle $T$ of $H$. Add
    $V(G_1)\setminus\{u,v\}$ to $X_T$. The resulting graph $H^+$
    contains $G$ as a spanning subgraph, and is in the same class as
    the original $H^+$.

  \item Suppose there is a $(2,3)$-separation $(G_1,G_2)$ of order 2:
    Without loss of generality, $a$ is the nominated vertex in
    $G_1-G_2$, $\{u,b\}=V(G_1\cap G_2)$, and $c$ and $d$ are the
    nominated vertices in $G_2-G_1$.  Let $G'$ be the graph obtained
    by contracting $G_1$ onto the edge $ub$, and nominating
    $u,b,c,d$. (This is possible since $G$ is 2-connected.)\

    If $G'$ contains a $\{u,b,c,d\}$-minor, then adding $G_1-b$ to the
    branch set containing $u$ gives an $\{a,b,c,d\}$-minor in $G$, and
    we are done. Otherwise, by the choice of $G$, $G'$ is a spanning
    subgraph of some obstruction $H^+$.  Since $ub$ is an edge of $G'$
    and both $u$ and $b$ are nominated in $G'$, $H^+$ is in class
    $\mathcal{A}$, $\mathcal{D}$ or $\mathcal{E}$.
    
    If $u$ is type-1, then $ub$ is in the outerface of $H$ (as
    embedded in Figure~\ref{fig:obstructions}).  Let $x$ be the
    neighbour of $u$ distinct from $b$ in this outerface. Add $a$ into
    the outerface of $H$ adjacent to $u,b,x$, and let
    $X_{a,u,b}:=V(G_1)\setminus\{a,b,u\}$. The resulting graph $H^+$
    is in the same class as the original $H^+$, and contains $G$ as a
    spanning subgraph.

    If $u$ is type-2, then $H^+\in \mathcal{A}$.  Add $a$ to $H$
    adjacent to $u$ and $b$, thus creating the triangle $aub$. Let
    $X_{aub}:=V(G_1)\setminus\{a,u,b\}$. The resulting graph $H^+$ is
    in class $\mathcal{E}$, and contains $G$ as a spanning subgraph.
    
    Now assume that $u$ is type-3. Thus $ub$ is in one triangle $ubx$
    in $H$ (since both $u$ and $b$ are nominated in $G'$).  In $H$,
    delete $u$, add $a$ adjacent to $x$ and $b$ creating the triangle
    $axb$, and let $X_{axb}:= V(X_{ubx}) \cup
    V(G_1)\setminus\{a,b\}$. The resulting graph $H^+$ contains $G$ as
    a spanning subgraph and is in the same class as the original
    $H^+$.

  \item Suppose there is a $(3,3)$-separation $(G_1,G_2)$ of order 2:
    Without loss of generality, $a\in V(G_1-G_2)$, $\{b,c\}=V(G_1\cap
    G_2)$, and $d\in V(G_2-G_1)$. Let $H:=K_4-ad$ where
    $V(H):=\{a,b,c,d\}$. Let $X_{abc}:=V(G_1)\setminus\{a,b,c\}$ and
    $X_{bcd}:=V(G_2)\setminus\{b,c,d\}$. The resulting graph $H^+$ is
    in class $\mathcal{D}$, and contains $G$ as a spanning subgraph.

  \item Suppose there is a $(2,2)$-separation $(G_1,G_2)$ of order 2:
    Let $\{u,v\}:=V(G_1\cap G_2)$. Let $G_i'$ be the graph obtained
    from $G_i$ by adding the edge $uv$. Since $G$ is 2-connected, by
    Lemma~\ref{lem:22}, if $G'_1$ contains an $\{a,b,u,v\}$-minor or
    $G'_2$ contains a $\{u,v,c,d\}$-minor, then $G$ contains an
    $\{a,b,c,d\}$-minor, and we are done. Otherwise, by the choice of
    $G$, each $G'_i$ is a spanning subgraph of an obstruction $H_i^+$.
    Since the nominated vertices $u$ and $v$ are adjacent in $G'_1$
    and $G'_2$, $H_1^+$ and $H_2^+$ are class $\mathcal{A}$,
    $\mathcal{D}$ or $\mathcal{E}$.

    Consider the case in which $H_1^+\in\mathcal{D}$. Then the edge
    $uv$ is either on the outerface of $H_1$ or is a diagonal of
    $H_1$. If $uv$ is a diagonal of $H_1$ then $H_1\cong K_4-ab$ since
    every triangle of $H_1$ is a face of $H_1$. Similarly, if
    $H_2^+\in\mathcal{D}$ and $uv$ is a diagonal of $H_2$, then
    $H_2\cong K_4-cd$.


    Let $H^+$ be the graph obtained by identifying $u,v$ in $H_1^+$
    with $u,v$ in $H_2^+$. Thus $H^+$ contains $G$ as a spanning
    subgraph.  By adding gray edges to $H^+$ as illustrated in
    Figure~\ref{fig:22}, we now show that $H^+$ is an
    obstruction. Consider the following cases:
    \begin{itemize}
    \item If $H_1^+\in \mathcal{A}$ and $H_2^+\in \mathcal{A}$ then
      $H^+\in \mathcal{C}$.

    \item Say $H_1^+\in \mathcal{A}$ and $H_2^+\in \mathcal{D}$.  If
      $uv$ is on the outerface of $H_2$ then $H^+\in
      \mathcal{E}$. Otherwise, $uv$ is a diagonal of $H_2$, and
      $H^+\in \mathcal{C}$.

    \item If $H_1^+\in \mathcal{A}$ and $H_2^+\in \mathcal{E}$ then
      $H^+\in \mathcal{F}$.

    \item Say $H_1^+\in \mathcal{D}$ and $H_2^+\in \mathcal{D}$.  If
      $uv$ is on the outerface of $H_1$ and $uv$ is on the outerface
      of $H_2$ then $H^+\in\mathcal{D}$.  If $uv$ is a diagonal of
      $H_1$ and $uv$ is on the outerface of $H_2$ then
      $H^+\in\mathcal{E}$.  Otherwise, $uv$ is a diagonal of $H_1$ and
      $uv$ is a diagonal of $H_2$, and $H^+\in\mathcal{B}$.

    \item Say $H_1^+\in \mathcal{E}$ and $H_2^+\in \mathcal{D}$.  If
      $uv$ is on the outerface of $H_2$ then $H^+\in
      \mathcal{E}$. Otherwise, $uv$ is a diagonal of $H_2$, and
      $H^+\in \mathcal{F}$.

    \item If $H_1^+\in \mathcal{E}$ and $H_2^+\in \mathcal{E}$ then
      $H^+\in \mathcal{F}$.

    \end{itemize}

    \begin{figure}
      \begin{center}
        \includegraphics[width=0.95\textwidth]{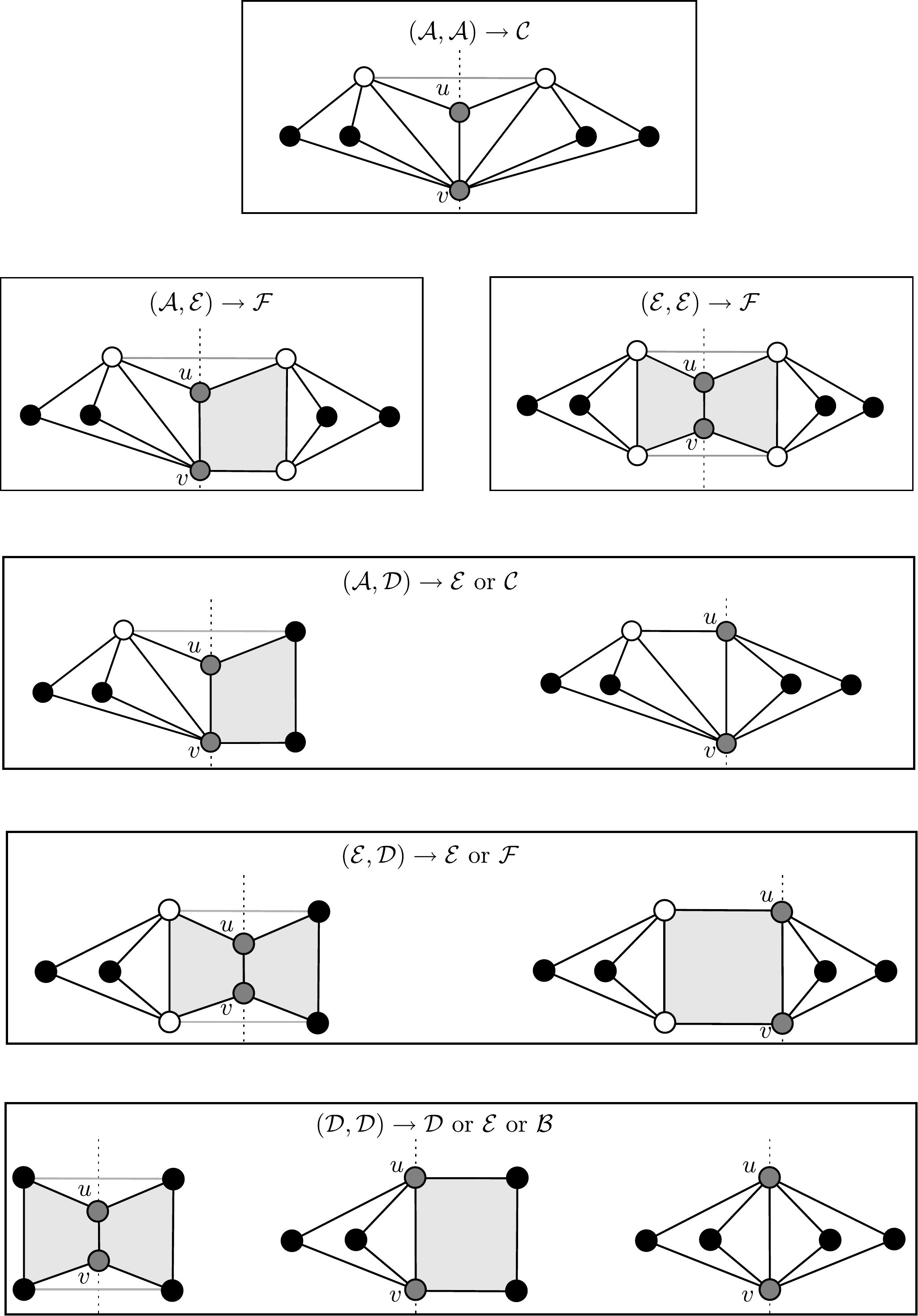}
      \end{center}
      \caption{\label{fig:22} Constructions of new obstructions in the
        case of a $(2,2)$-separation. Black vertices are
        nominated. Gray vertices are the cut-pair. White vertices are
        not nominated. Gray edges are inserted. Gray regions are
        webs.}
    \end{figure}

  \end{itemize}

  Now assume that $G$ is 2-connected and every separation of order 2
  is a $(1,3)$-separation. Before addressing this case it will be
  convenient to first eliminate a particular separation of order 3.

  \begin{itemize}

  \item Suppose there is a separation $(G_1,G_2)$ of order 3 with no
    nominated vertices in $G_2-G_1$, such that $|V(G_2)|\geq 5$:

    Let $\{u,v,w\}:=V(G_1\cap G_2)$. We claim that $G_2$ contains a
    $\{u,v,w\}$-minor. If not, then by Lemma~\ref{lem:K3}, there is a
    vertex $x$ such that at most one of $u,v,w$ is in each component
    of $G_2-x$. Since $|V(G_2)|\geq 5$ there is a vertex $y\in
    V(G_2)\setminus\{u,v,w,x\}$. If $y$ is in the same component of
    $G_2-x$ as $u$, then $\{u,x\}$ is a cut-pair that forms a
    $(0,4)$-separation of order 2 in $G$. Thus $y$ is not in the same
    component of $G_2-x$ as $u$.  Similarly, $y$ is not in the same
    component of $G_2-x$ as $v$ or $w$.  Thus $x$ is a cut-vertex,
    which is a contradiction.  Hence $G_2$ contains a
    $\{u,v,w\}$-minor.  Let $G'$ be the graph obtained from $G_1$ by
    adding the triangle $uvw$. Thus $G'$ is a minor of $G$, and
    $|V(G')|<|V(G)|$. If $G'$ contains an $\{a,b,c,d\}$-minor then so
    does $G$ and we are done.  Otherwise, by the choice of $G$, $G'$
    is a spanning subgraph of an obstruction $H^+$. The triangle $uvw$
    is contained in $T\cup X_T$ for some triangle $T$ of $H$.  Add
    $V(G_2)\setminus\{u,v,w\}$ to $X_T$. The resulting graph $H^+$
    contains $G$ as a spanning subgraph (since the neighbours of each
    vertex in $G_2\setminus\{u,v,w\}$ are in $G_2$) and is of the same
    class as the original $H^+$.

  \end{itemize}

  Now assume that if $(G_1,G_2)$ is a separation of order 3 with no
  nominated vertices in $G_2-G_1$, then $|V(G_2)|=4$.  We consider the
  following two types of $(1,3)$-separations.

  \begin{itemize}

  \item Suppose there is a $(1,3)$-separation $(G_1,G_2)$ of order 2,
    such that $|V(G_1)|\geq 4$, or $|V(G_1)|=3$ and $G_1\not\cong
    K_3$:

    Let $a$ be the nominated vertex in $G_1-G_2$.  Let
    $\{u,v\}:=V(G_1\cap G_2)$. Let $G'$ be the graph obtained from
    $G_2$ by adding the edge $uv$ if it does not already exist, and by
    adding a new vertex $a'$ adjacent to $u$ and $v$, where $a',b,c,d$
    are nominated in $G'$. Observe that $|V(G')|<|V(G)|$ or if
    $|V(G')|=|V(G)|$ then $|E(G')|>|E(G)|$. Thus by the choice of $G$,
    $G'$ contains an $\{a',b,c,d\}$-minor, or $G'$ is a spanning
    subgraph of an obstruction $H^+$.

    First suppose that $G'$ contains a $K_4$-minor $A',B,C,D$
    respectively rooted at $a',b,c,d$.  Since $a'$ has degree $2$ in
    $G'$, without loss of generality, $u$ is in $A'$. Now $G_1-v$ is
    connected, as otherwise $v$ is a cut-vertex in $G$. Thus
    $A:=(G_1-v)\cup A'$ is connected and is disjoint from $B\cup C\cup
    D$.  We claim that $A,B,C,D$ is an $\{a,b,c,d\}$-minor in $G$.
    Clearly $A,B,C,D$ respectively contain $a,b,c,d$. Since the edge
    $uv$ was added to $G'$, it may be that $G'$ is not a minor of
    $G$. So this claim is not immediate. However, if $uv$ is in $G$
    then $G'$ is a minor of $G$, and $A,B,C,D$ is a $K_4$-minor in
    $G$, and we are done.  It remains to show that the edge $uv$ is
    not needed for $A,B,C,D$ to be a $K_4$-minor. Since $u$ is in $A$,
    and $A$ is connected, the only problem is if $uv$ is the only edge
    between $A$ and some other branch set, say $B$. But, since $G$ is
    2-connected, $v$ has a neighbour in $G_1-u-v$, which is a subgraph
    of $A$. This proves that $A,B,C,D$ is an $\{a,b,c,d\}$-minor in
    $G$.

    Now assume that $G'$ is a spanning subgraph of some obstruction
    $H^+$. Thus $a',u,v\in T\cup X_T$ for some triangle $T$ of $H$,
    and $a'\in T$. Rename $a'$ as $a$ in $H$, and add
    $V(G_1)\setminus\{a,u,v\}$ to $X_T$. The resulting graph $H^+$ is
    in the same class as the original $H^+$ and contains $G$ as a
    spanning subgraph.

  \end{itemize}

  Now assume that if $(G_1,G_2)$ is a separation of order 2, then
  $|V(G_1)|=3$, the vertex in $G_1-G_2$ is nominated, and $G_1\cong
  K_3$ (since $G$ is 2-connected).

  \begin{itemize}

  \item Suppose there is a $(1,3)$-separation $(G_1,G_2)$ of order 2:
    Let $a$ be the nominated vertex in $G_1-G_2$. Let
    $\{u,v\}:=V(G_1\cap G_2)$. Thus $G_1\cong K_3$ with vertex set
    $\{a,u,v\}$.

    Let $G_u$ be the graph obtained from $G$ by contracting the edge
    $au$ into $u$, and nominating $u$.  Let $G_v$ be the graph
    obtained from $G$ by contracting the edge $av$ into $v$, and
    nominating $v$.  Each of $G_u$ and $G_v$ have four nominated
    vertices. Since $a$ has degree 2 in $G$, $G$ contains an
    $\{a,b,c,d\}$-minor if and only if $G_u$ contains a
    $\{u,b,c,d\}$-minor or $G_v$ contains a $\{v,b,c,d\}$-minor.  Also
    observe that $G_u\cong G_v$; they only differ in one nominated
    vertex.  For the time being, concentrate on $G_u$; we will return
    to $G_v$ later.

    If $G_u$ contains a $\{u,b,c,d\}$-minor, then $G$ contains an
    $\{a,b,c,d\}$-minor, and we are done. Otherwise, by the choice of
    $G$, $G_u$ is a spanning subgraph of an obstruction $H^+$. Since a
    class $\mathcal{A}$ obstruction has a $(2,3)$-separation, and a
    class $\mathcal{B},\mathcal{C},\mathcal{E}$ or $\mathcal{F}$
    obstruction has a $(2,2)$-separation, $H^+$ is in class
    $\mathcal{D}$.

    If $|X_T|\geq2$ for some triangle $T$ of $H$ then $(G-X_T,T\cup
    X_T)$ is a separation of order 3 with no nominated vertices in
    $X_T$, such that $|V(T\cup X_T)|\geq 5$, which is a
    contradiction. Thus $|X_T|\leq 1$. If $X_T=\{w\}$ then move $w$
    out of $X_T$ into $H$; the resulting graph $H^+$ is in
    $\mathcal{D}$ and contains $G_u$ as a spanning subgraph. Repeat
    this step until $X_T=\emptyset$ for each triangle $T$ of $H$. Thus
    $G_u$ is a spanning subgraph of $H$ (not $H^+$), and $G_u$ is
    planar. Since $G_u$ was obtained from $G$ by deleting a degree-2
    vertex whose neighbours are adjacent, $G$ is also planar.



    Since $H\in \mathcal{D}$, $u$ is type-1. Let $S$ be the set of
    degree-2 nominated vertices in $G$. Thus $a\in
    S\subseteq\{a,b,c,d\}$. Observe that $G$ is almost 3-connected in
    the sense that the only cut-pairs are the neighbours of vertices
    in $S$, and in this case the cut-pair are adjacent. As illustrated
    in Figure~\ref{fig:onethree}, let $G^*:=G-S$. A separation in
    $G^*$ is a separation in $G$. Thus $G^*$ is 3-connected and
    planar. Hence $G^*$ has a unique planar embedding.  Moreover,
    every planar embedding of $G$ is obtained from the unique planar
    embedding of $G^*$ by drawing each vertex $x\in S$ in one of the
    two faces that contain the edge between the two neighbours of $x$.
    In the planar embedding of $G_u$ induced by the planar embedding
    of $H$, the nominated vertices $u,b,c,d$ are on the
    outerface. Moreover, the unique planar embedding of $G^*$ is
    obtained from this embedding of $G_u$ by deleting
    $S\setminus\{a\}$.

    \begin{figure}
      \begin{center}
        \includegraphics[width=0.8\textwidth]{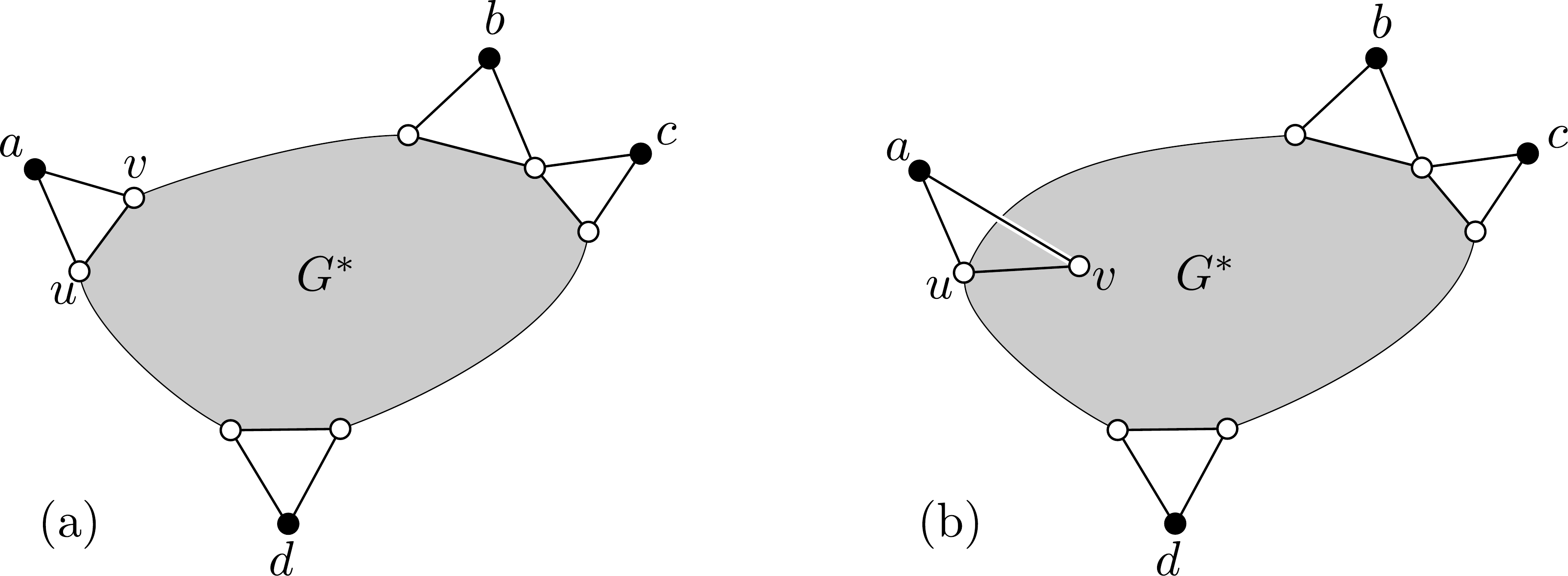}
      \end{center}
      \caption{\label{fig:onethree} Illustration of $G$ with a
        $(1,3)$-separation of order 2. Vertex $a$ has degree 2, and
        $b,c,d$ might have degree 2.}
    \end{figure}

    If the edge $uv$ is on the outerface of $G_u$ (as in
    Figure~\ref{fig:onethree}(a)), then draw $a$ in the outerface of
    $G_u$ adjacent to $u$ and $v$, and possibly add edges between $a$
    and other nominated vertices to obtain an obstruction (in the same
    class as $H$) that contains $G$ as a spanning subgraph.

    Now assume that $uv$ is not on the outerface of $G_u$ (as in
    Figure~\ref{fig:onethree}(b)). Recall that $G_u\cong G_v$, and
    $v,b,c,d$ are nominated in $G_v$.  Consider this embedding of
    $G_u$ to be an embedding of $G_v$. The outerface of $G_v$ contains
    $b,c,d$ but not $v$.

    For $x\in\{b,c,d\}$, if $x\in S$ then choose a neighbour $x'$ of
    $x$, otherwise let $x':=x$.  If $x$ and $y$ are distinct vertices
    in $S$, then $N_G(x)\neq N_G(y)$, as otherwise $G$ would contains
    a $(2,2)$-separation of order $2$.  Thus we may choose $b',c',d'$
    so that they are distinct.  Each of $b',c',d'$ are on the
    outerface of $G_v$. So $v,b',c',d'$ are all distinct.

    Consider $v,b',c',d'$ to be nominated vertices in $G^*$.  Consider
    the embedding of $G^*$ formed from $H$.  Then $b',c',d'$ are on
    the outerface of $G^*$, but $v$ is not.  In a 3-connected planar
    graph, three vertices all appear on at most one face.  Thus, no
    face of $G^*$ contains all of $v,b',c',d'$. Thus by
    Theorem~\ref{thm:3ConPlanar}, $G^*$ contains a
    $\{v,b'c',d'\}$-minor. Given that $G^*$ can be obtained from $G$
    by contracting $av$, $bb'$, $cc'$ and $dd'$, $G$ contains an
    $\{a,b,c,d\}$-minor. (Here, if $b=b'$ then contracting $bb'$ does
    nothing.)\
  \end{itemize}

  Now assume that $G$ is 3-connected. The result follows from Theorem
  \ref{thm:3con}, since a web is in class $\mathcal{D}$.
\end{proof}

\section{Algorithmics}

\citet{RS-GraphMinorsXIII-JCTB95} presented a $O(n^3)$ time algorithm
that (for fixed $t$) tests whether a given $n$-vertex graph contains a
$K_t$-minor rooted at $t$ nominated vertices. We conjecture that for
$t=4$ there is a $O(n)$ time algorithm for this problem; see
\citep{KLR,LR,Hagerup,Woeginger} for related linear time algorithms.


\end{document}